\documentclass[11pt]{article}
\usepackage[utf8]{inputenc}
\usepackage[english]{babel}
\usepackage{amsthm}
\usepackage{amssymb}
\usepackage{amsmath}
\usepackage{graphicx}
\usepackage {bibnames}                  
\usepackage{enumitem}
\usepackage{mathtools}
\usepackage{url}

\newtheorem{theorem}{Theorem}[section]
\newtheorem{proposition}[theorem]{Proposition}
\newtheorem{definition}[theorem]{Definition}
\newtheorem{lemma}[theorem]{Lemma}
\newtheorem{corollary}[theorem]{Corollary}

\newtheorem{remark}[theorem]{Remark}

\title{On the capacity dimension of the boundary of CAT(0) spaces}
\author{Dawei Wang}
\date{}

\begin{document}
	
	\maketitle

\begin{abstract}
	In this paper, we study the capacity dimension of the boundary of $CAT(0)$ spaces. We first compare the two metrics on the boundary of a hyperbolic $CAT(0)$ space, i.e., the visual metric and the conical metric, and prove that they give the same capacity dimension of the boundary. Then we study the capacity dimension of the boundary of buildings, which is an important class of $CAT(0)$ spaces. Finally, we give a possible method to prove the finiteness of the asymptotic dimension of $CAT(0)$ spaces.
\end{abstract}

\section {Introduction}

Asymptotic dimension is one of the most interesting invariants in large-scale geometry of metric spaces and in particular finitely generated groups. It is important because the Novikov conjecture holds for groups with finite asymptotic dimension, see \cite{Yu}. It is known that the asymptotic dimension of $\delta$-hyperbolic groups is finite, which was proved by Gromov in \cite{Gromov}. However, the finiteness of the asymptotic dimension of CAT(0) groups has been open for decades.        

To get a more precise bound on the asymptotic dimension of hyperbolic space, Buyalo introduced the capacity dimension \cite{B}.

\begin{definition}[Capacity dimension]
	Let $X$ be a metric space. We say the capacity dimension is at most $n$, denoted by $cdim X\le n$, if there exists $0<c\le1$ and $\lambda_0$, such that for all $0<\lambda\le\lambda_0$, there is a cover $\mathcal{U}$ with $order(\mathcal{U})\le n+1, mesh(\mathcal{U})\le\lambda$, and $\mathcal{L}(\mathcal{U})\ge c\lambda$. The number $\frac{\mathcal{L}(\mathcal{U})}{mesh(\mathcal{U})}\ge c$ is called the capacity of $\mathcal{U}$.
\end{definition}

With the capacity dimension, Buyalo and Lebedeva proved the following theorem.
\begin{theorem}[\cite{B}]\label{asdimcdiminequality}
	Let $X$ be a visual Gromov hyperbolic space. Then
	$$asdim X\le cdim \partial X +1.$$
\end{theorem}  

The inequality gives a new point of view to understand the asymptotic dimension: by looking at the large-scale geometry captured in the boundary. There are different ways to define a metric on the boundary of CAT(0) spaces. However, none of them is as good as the visual metric for $\delta$-hyperbolic spaces. In \cite{Moran2}, Moran proved some good properties of a particular class of metric on the boundary of CAT(0) spaces, the conical metric, see definition \ref{Moranmetric}. In particular, with this metric, she proved that the capacity dimension of the boundary of CAT(0) groups is finite. We believe the conical metric is the right metric for the boundary of CAT(0) spaces to study the large-scale geometry.

Now the ultimate goal is to prove that the asymptotic dimension of CAT(0) groups is finite. With Moran's result, the question is: can we get inequality similar to the one in Theorem \ref{asdimcdiminequality}.

In this paper, we try to give a partial answer to the question above. In particular, we first try to understand the conical metric better by proving the following theorem.
\begin{theorem}
	Let X be a cobounded $\delta$-hyperbolic CAT(0) proper geodesic space. Let $cdim_v(\partial X)$ be the capacity dimension of the boundary of $X$ with the visual metric and $cdim_c(\partial X)$ be the capacity dimension with the visual metric. Then we have $cdim_v(\partial X) = cdim_c(\partial X)$. In particular, for a hyperbolic CAT(0) group G, we have $cdim_v(\partial G) = cdim_c(\partial G)$.
\end{theorem}

Then we study the capacity dimension of the boundary of nonspherical buildings, which is an important class of CAT(0) spaces.

\begin{theorem}
	The capacity dimension of the boundary of any nonspherical building is equal to the capacity dimension of the boundary of an apartment in the building.
\end{theorem}

In particular, with the result on the asymptotic dimension of buildings in \cite{DS}, we have the following equality.
\begin{corollary}
	Let $X$ be a Euclidean or hyperbolic building. Then 
	$$asdim(X)=cdim\partial X +1.$$
\end{corollary}
\begin{remark}
	In fact, the Corollary is true as long as the equality holds in the apartment of the building.
\end{remark} 

\section {Preliminaries}
In this section, we will introduce the preliminaries that we will use.

\subsection{$\delta$-hyperbolic spaces and the visual metric}
The following introduction to $\delta$-hyperbolic spaces is from \cite{GH}. Another good reference is \cite{BH}.

Throughout this section, we assume $X$ is a proper geodesic metric space. We denote the distance between two points $y, z\in X$ by $|y-z|$ or $d(y,z)$.
\begin{definition}\label{gromovproduct}
	Given a base point $x\in X$, the Gromov Product of two points $y, z\in X$ is defined by
	$$(y|z)_x = \frac12( |y-x| + |z-x| - |z-y| ).$$
	We write $(y|z)$ if there is no ambiguity about the base point. 
\end{definition}

We can extend the definition of the Gromov product to the boundary $\partial X$.

\begin{definition}
	Let $X$ be a proper hyperbolic space. For any $a, b \in \partial X$, the Gromov Product in $\partial X$ is defined by
	$$(a|b) = \sup\liminf_{i,j\rightarrow\infty}(x_i|y_j)$$
	where the supremum is taken over all sequences $(x_i)_{i\ge 1}$ tending towards $a$ and $(y_j)_{j\ge 1}$ tending towards $b$.
\end{definition} 

\begin{proposition}\label{liminf}
	Let $X$ be a proper $\delta$-hyerpoblic space and $a,b\in \partial X$. Then for all sequences $x_i\rightarrow a$ and $y_j\rightarrow b$, we have 
	$$(a|b)-2\delta\le \liminf_{i,j\rightarrow\infty}(x_i|y_j)\le(a|b).$$
\end{proposition}
Now we will construct a metric on the boundary $\partial X$. Fix a real number $\epsilon>0$, for any $a, b\in \partial X$, define
$$\rho_\epsilon(a,b)=e^{-\epsilon(a|b)}.$$
However, $\rho_\epsilon$ may not define a metric. We will modify $\rho_\epsilon$ to define a distance on $\partial X$. 

A chain between $a$ and $b$ in $\partial X$ is a finite sequence $a=a_0, a_1,\ldots, a_n=b$ of points in $\partial X$. Denote the set of chains between $a$ and $b$ by $C_{a,b}$, and let
$$\rho_\epsilon(a_0,a_1,\ldots,a_n) = \sum_{i=1}^{n}\rho_\epsilon(a_{i-1},a_i).$$
Then define 
$$d_{\epsilon}(a,b) = \inf\{\rho_\epsilon(c) : c\in C_{a,b}\}.$$
It turns out $d_{\epsilon}$ is a metric on $\partial X$.
\begin{proposition}\label{visualmetric}\cite[Proposition 10]{GH}
	Fix $\epsilon$ and let $\epsilon' = e^{\epsilon\delta}-1$. If $\epsilon'\le\sqrt{2}-1$, then $d_\epsilon$ is a distance on $\partial X$ and we have
	$$(1-2\epsilon')\rho_\epsilon(a,b)\le d_\epsilon(a,b)\le \rho_\epsilon(a,b)$$
	for all $a,b\in\partial X$. 
\end{proposition}
The metric $d_\epsilon$ is called the visual metric on $\partial X$, we may also denote it by $d_v$.

%


\subsection{CAT(0) Space and the conical metric}
We will introduce the conical metric on the boundary of CAT(0) spaces. For a detailed introduction of CAT(0) space and its boundary, please see \cite{BH}. There are various ways to define metrics on $\partial X$. We will define the one that works best for our purpose. It was first introduced by D. Osajda and was used by D. Osajda and J. Swiatkowski in \cite{OS}. Later, M. A. Moran proved some important properties in \cite{Moran2} that make this metric significant. See \cite{Moran} for more details of this metric.

\begin{definition}[The conical metric]\label{Moranmetric}
	Let $X$ be a proper CAT(0) space. Fix a basepoint $x_0$ and choose $A>0$. For $[\gamma], [\gamma']\in \partial X$, let $\gamma,\gamma':[0,\infty)\rightarrow X$ be the geodesic rays based at $x_0$ that represent $[\gamma], [\gamma']$ respectively. Let $t\in (0,\infty)$ be such that $d(\gamma(t),\gamma'(t))=A$. If such $t$ does not exist set $t=\infty$. Then, define $d_{A,x_0}:\partial X\times\partial X\rightarrow \mathbb{R}$ by
	$$d_{A,x_0}([\gamma],[\gamma'])=\frac1t.$$
\end{definition}

\begin{lemma}\cite[Lemma 3.3.1]{Moran2}
	Let $X$ be a CAT(0) space with base point $x_0$, then for any $A>0$, $d_{A,x_0}$ is a metric on $\partial X$.
\end{lemma}

\begin{lemma}
	The topology induced by the $d_{A,x_0}$ metric on $\partial X$ is equivalent to the visual topology $\mathcal{T}_{x_0}$ on $\partial X$.
\end{lemma}


\begin{lemma}
	Let $X$ be a proper CAT(0) space. For any $A,A'>0$, the identity map on the boundary $id_{\partial X}: (\partial X, d_{A,x_0})\rightarrow (\partial X, d_{A',x_0})$ is a quasi-symmetry.
\end{lemma}

\begin{lemma}
	Suppose $X$ is a complete CAT(0) space. For any $x_0,x_0'\in X$, the identity map on the boundary $id_{\partial X}: id_{\partial X}: (\partial X, d_{A,x_0})\rightarrow id_{\partial X}: (\partial X, d_{A,x_0'})$ is a quasi-symmetry.
\end{lemma}

The following theorem shows that the group of isometries of a CAT(0) space has a "nice" action on the boundary.
\begin{theorem}\cite[Theorem 3.1.5]{Moran2}
	Suppose $G$ is a finitely generated group that acts by isometries on a complete CAT(0) space $X$. Then the induced action of $G$ on $(\partial X, d_{A,x_0})$ is a quasi-symmetry. In other words, $G$ acts by quasi-symmetries on $\partial X$.
\end{theorem}

There is a simple geometric property for CAT(0) spaces that we will use repeatedly. See the proof in\cite{Moran}.
\begin{lemma}\label{ratio}
	Let $(X,d)$ be a CAT(0) space and suppose $\gamma, \ \gamma':[0,\infty)\rightarrow X$ are two geodesic rays based at the same point $p\in X$. Then for $0<s\le t<\infty$, we have 
	$$d(\gamma(s),\gamma'(s))\le\frac st d(\gamma(t),\gamma'(t)).$$
\end{lemma}


\subsection{Dimension Theory}
In this section, we review some dimension theories that play important roles in geometric group theory. We first review some terminology.
\begin{definition}
	Let $X$ be a metric space and $\mathcal{U}$ be a cover of $X$. We define the $order(\mathcal{U})$ to be the smallest integer $n$ such that each $x\in X$ is contained in at most $n$ elements of $\mathcal{U}$. We define $mesh(\mathcal{U})=sup\{diam(U)|U\in\mathcal{U}\}$. We say the cover $\mathcal{U}$ is uniformly bounded if there exists some $\lambda$ such that $mesh(\mathcal{U})\le\lambda$. The Lebesgue number of $\mathcal{U}$ is defined as $\mathcal{L}(\mathcal{U})=\inf_{x\in X}\mathcal{L}(\mathcal{U},x)$, where $\mathcal{L}(\mathcal{U},x)=\sup_{x\in X}\{d(x,X-U)|U\in \mathcal{U}\}$.
\end{definition}

Now we introduce the topological dimension, also called the covering dimension. See for example \cite{M} for more details. Recall that a refinement of a cover $C$ of a topological space $X$ is a new cover $D$ of $X$ such that every set in $D$ is contained in some set in $C$.
\begin{definition}[Topological dimension]
	Let $X$ be a topological space. We say the topological dimension of $X$ is at most $n$, denoted by $dim X\le n$, if every open cover of $X$ has an open refinement of order at most $n+1$.
\end{definition}
For a compact metric space, the topological dimension has an equivalent definition.
\begin{definition}\label{topologicaldimension}
	Let $X$ be a compact metric space. Then $dim X\le n$ if for every (small) $\lambda>0$, there is an open cover $\mathcal{U}$ with $Mesh(\mathcal{U})\le\lambda$ and order at most $n+1$.
\end{definition}
An important dimension in geometric group theory is the asymptotic dimension.
\begin{definition}[Asymptotic dimension]
	Let $X$ be a metric space. We say that $asdim(X)\le n$ if for any (large) $\lambda$ there is a uniformly bounded cover $\mathcal{U}$ of $X$ with $order(\mathcal{U})\le n+1$ and $\mathcal{L}(\mathcal{U})\ge \lambda$.
\end{definition}
The asymptotic dimension is important because of its relation with the Novikov conjecture, see \cite{Yu} for details. 

Another important dimension is the capacity dimension or linearly-controlled dimension. It was first introduced by Buyalo in \cite{B}.
\begin{definition}[Capacity dimension]
	Let $X$ be a metric space. We say the capacity dimension is at most $n$, denoted by $cdim X\le n$, if there exists $0<c\le1$ and $\lambda_0$, such that for all $0<\lambda\le\lambda_0$, there is a cover $\mathcal{U}$ with $order(\mathcal{U})\le n+1, mesh(\mathcal{U})\le\lambda$, and $\mathcal{L}(\mathcal{U})\ge c\lambda$. The number $\frac{\mathcal{L}(\mathcal{U})}{mesh(\mathcal{U})}\ge c$ is called the capacity of $\mathcal{U}$.
\end{definition}

The capacity dimension has many equivalent definitions. Here is the one that we will also use. Recall that two sets $U$ and $U'$ are L-disjoint if $$d(U,U')=\inf\{d(x,y)|x\in U, y \in U'\}\ge L.$$
\begin{definition}
	We say $cdim(X)\le n$ if there exists $c>0$ such that for any sufficiently small $L$, there are $n+1$ families of $L-$disjoint sets that cover $X$ and are $cL$ bounded.
\end{definition}

Gromov observed that all hyperbolic groups have finite asymptotic dimension. Buyalo established a more precise bound on the asymptotic dimension.
\begin{theorem}[\cite{B}]\label{asdimlecdim}
	Let $X$ be a visual Gromov hyperbolic space. Then
	$$asdim X\le cdim \partial X +1.$$
\end{theorem}
In case of hyperbolic groups, Buyalo and Lebedeva proved the following equalities relating the three dimensions.
\begin{theorem}[\cite{BL}]\label{topequalcap}
	For any hyperbolic group $G$, we have 
	$$asdimG = cdim\partial G +1 = dim\partial G +1.$$
\end{theorem}

For CAT(0) spaces, Moran proved the following result with the conical metric $d_c$ on the boundary.
\begin{theorem}\cite[Theorem 3.2.1]{Moran2}
	Suppose $G$ acts geometrically on a proper CAT(0) space $X$. Then $cdim(\partial X, d_c)<\infty$.
\end{theorem}
However, the inequality in Theorem \ref{asdimlecdim} for CAT(0) spaces is not known.


\subsection{Buildings}

The following introduction of buildings is from \cite{AB}.
\begin{definition}
	We say that $W$ is a Coxeter group and $(W,S)$ is a Coxeter system if $W$ admits the presentation
	$$\left< S; (st)^{m(s,t)}=1\right>,$$
	where $m(s,t)$ is the order of $st$ and there is one relation for each pair $s,t$ with $m(s,t)<\infty$.
\end{definition}
%

Fix a Coxeter system $(W,S)$ and denote by $l=l_S$ the length function on $W$ with respect ot $S$.
\begin{definition}
	A building of type $(W,S)$ is a pair $(\Delta,\delta)$ consisting of a nonempty set $\Delta$, whose elements are called chambers, together with a map $\delta: \Delta\times\Delta\rightarrow W$, called the Weyl distance function, such that for all $C,D\in \Delta$, the following three conditions hold:\\
	({\bf WD1}) $\delta(C,D)=1$ if and only if $C=D$.\\
	({\bf WD2}) If $\delta(C,D)=w$ and $C'\in \mathcal{C}$ satisfies $\delta(C',C)=s\in S$, then $\delta(C',D)=sw$ or $w$. If in addition, $l(sw)=l(w)+1$, then $\delta(C',D)=sw$.\\
	({\bf WD3}) If $\delta(C,D)=w$, then for any $s\in S$ there is a chamber $C'\in \Delta$ such that $\delta(C',C)=s$ and $\delta(C',D)=sw$.
\end{definition}

\begin{definition}
	A nonempty subset $\mathcal{M}$ of $\Delta$ is called thin (resp. thick) if $\mathcal{P}\cap\mathcal{M}$ has cardinality $2$ (resp. $>2$) for every panel $\mathcal{P}$ of $\Delta$ with $\mathcal{P}\cap\mathcal{M}\neq\Phi$. A thin subbuilding of $\Delta$ is called an apartment of $\Delta$.
\end{definition}


\begin{proposition}
	For any two chambers $C,D\in \Delta$, there exists an apartment $\Sigma$ of $\Delta$ with $C,D\in \Sigma$.
\end{proposition}

Davis proved in \cite{D} that with the corrected defined metric, all buildings are CAT(0). This is called the Davis realization. See also \cite{AB} for detailed construction.

\begin{theorem}\cite[Theorem 12.66]{AB}
	For any building $\Delta$, its Davis realization $X=Z(\Delta)$ is a complete CAT(0) space.
\end{theorem}


For simplicity, we will abuse the notation and let $\Delta$ and $\Sigma$  be the Davis realization of the building and the apartment respectively. Also, we will use $d$ as the metric in Davis realization. 

There is an important retraction from the building to the apartment.

\begin{proposition}
	Every apartment is a retract of $\Delta$.
\end{proposition}

\begin{definition}
	Given an apartment $\Sigma$ and a chamber $C\in\Sigma$, there is a canonical retraction $\rho=\rho_{\Sigma,C}: \Delta \rightarrow \Sigma$. It is called the retraction onto $\Sigma$ centered at $C$. It can be characterized as the unique chamber map $\Delta\rightarrow \Sigma$ that fixes $C$ pointwise and maps every apartment containing $C$ isomorphically onto $\Sigma$.
\end{definition}

\begin{proposition}\label{prop1} Let $\rho=\rho_{\Sigma,C}$ be the apartment retraction. Then
	$$d(\rho(x),\rho(y))\le d(x,y)$$
	for all $x,y\in \Delta$, with equality if $x\in C$.
\end{proposition}

\section{Capacity dimensions with the visual metric and the conical metric}
In this section, we will prove the following theorem.
\begin{theorem}
	Let X be a cobounded $\delta$-hyperbolic CAT(0) proper geodesic space. Let $cdim_v(\partial X)$ be the capacity dimension of the boundary of $X$ with the visual metric and $cdim_c(\partial X)$ be the capacity dimension with the conical metric. Then we have $cdim_v(\partial X) = cdim_c(\partial X)$. In particular, for a hyperbolic CAT(0) group G, we have $cdim_v(\partial G) = cdim_c(\partial G)$.
\end{theorem}

We will need the following lemmas.
\begin{lemma}\label{lem1}
	Let $\displaystyle f(x)=\frac1{-a\ln x+b}$, where $a,b>0$. If $\displaystyle0<x<e^{-2}$, then $f''(x)<0$. Consequently, when $0<x<e^{-2}$, we have 
	\begin{align*}
	&f(cx)\ge cf(x)\quad \text{and}\\
	&f(cx)+f\big((1-c)x\big)\ge f(x),\quad \text{when}\quad 0<c<1.
	\end{align*}
\end{lemma}	
\begin{proof}
	Note that 
	$$f''(x)=\frac{a}{(-a\ln x +b)^3x^2}\cdot (2a+a\ln x-b).$$
	When $0<x<e^{-2}$, we have $-a\ln x+b>0$ and $x^2>0$, hence
	$$(-a\ln x+b)^3x^2>0.$$
	In addition, $\ln x<-2$, hence
	$$2a+a\ln x-b<-b.$$
	Therefore,
	$$f''(x)<\frac{a}{(-a\ln x+b)^3x^2}\cdot(-b)<0.$$
	Notice that $\lim_{x\rightarrow0}f(x)=0$, hence the first inequality follows from the Jensen's inequality for concave function
	$$f(cx+(1-c)y)\ge cf(x)+(1-c)f(y)\quad \text{for}\quad t \in (0,1).$$
	Applying the first inequality twice with $c$ and $1-c$ gives the second inequality.
\end{proof}

\begin{remark}
	The upper bound $e^{-2}$ for the equalities to hold is not sharp. However, it's enough to prove the main theorem.
\end{remark}

The following lemma is well-known.

\begin{lemma}\label{claim}
	Let $X$ be a hyperbolic CAT(0) space. Then for any geodesic rays $a, b:[0,\infty)\rightarrow X$ starting at the basepoint, we have
	\begin{enumerate}[label=(\arabic*)]
		\item $d(a(t),b(t))$ is non-decreasing with respect to $t$.\label{claim1}
		\item $(a(t)|b(t))$ is non-decreasing with respect to $t$.\label{claim2}
	\end{enumerate}
\end{lemma}

%
%
%
%

\begin{lemma}\label{sR}
	Let $X$ be a $\delta$-hyperbolic CAT(0) space. For any $s$ arbitrarily small, there exists $R>0$ such that for any $a,b\in\partial X$ with $a(t)$ and $b(t)$ being the representing geodesic rays starting at the basepoint, and any $R_1\ge R$, we have 
	\begin{equation}\label{1}
	(a|b)-2\delta-s<R_1-\frac12d\left(a(R_1),b(R_1)\right)<(a|b).
	\end{equation}
\end{lemma}

\begin{proof}
	Let $\{t_i\}$ be a sequence going to $\infty$. For any $a,b\in \partial X$, let $a(t)$ and $b(t)$ be the representing geodesic rays starting at the basepoint. By Lemma \ref{claim}\ref{claim2} we know
	$$\liminf_{i}(a(t_i)|b(t_i))=\lim_i(a(t_i)|b(t_i))=\lim_t(a(t)|b(t)).$$
	By Proposition \ref{liminf} and the equality above, we have 
	
	$$(a|b)-2\delta\le \lim_t(a(t)|b(t))\le(a|b).$$
	
	Since $\partial X$ is compact, and by Lemma \ref{claim}\ref{claim1}, for any $s$ arbitrarily small, there exists $R>0$ such that for any $a,b\in\partial X$ and any $R_1\ge R$, we have
	$$\lim_t(a(t)|b(t))-s<(a(R_1)|b(R_1))<\lim_t(a(t)|b(t)).$$
	
	Notice $(a(R_1)|b(R_1))=R_1-\frac12d\left(a(R_1),b(R_1)\right)$. Combining the above two inequalities gives
	\begin{equation*}
	(a|b)-2\delta-s<R_1-\frac12d\left(a(R_1),b(R_1)\right)<(a|b).
	\end{equation*}
\end{proof}

Recall that $A$ is a fixed number in the definition of the conical metric $d_c$. Also recall that in the definition of $d_v$, we fix $\epsilon$ and $\epsilon'=e^{\epsilon\delta}-1$ such that $\epsilon'\le \sqrt2-1$. Then by Proposition \ref{visualmetric} we know
$$(1-2\epsilon')\rho_\epsilon(x,y)\le d_\epsilon(x,y)\le\rho_\epsilon(x,y),$$
where $d_\epsilon$ is the visual metric $d_v$ with parameter $\epsilon$ and $\rho_\epsilon(x,y)=e^{-\epsilon(x|y)}$.

\begin{proposition}\label{lem2}
	Let $X$ be a $\delta$-hyperbolic CAT(0) space. Then there exist constants $a, b, k, B>0$, depending on $ A$, $\epsilon$ and $\delta$, such that for any $x, y\in\partial X$ with $d_v(x,y)\le B$, we have
	$$\frac1{-a\ln d_v(x,y)+b}<d_c(x,y)<\frac1{-a\ln d_v(x,y)+b-k}.$$

	
\end{proposition}

\begin{remark}\label{abk}
	In the proposition above, $a=\frac{1}{\epsilon}$, $b=\frac A2$, $k=2\delta+1-\frac1\epsilon\ln(1-2\epsilon')$ and 	$B=(1-2\epsilon')e^{-\epsilon(R-\frac12A+2\delta+1)}$, where $R$ is as in Lemma \ref{sR}.
\end{remark}

\begin{proof}

	Fix $A$ for the conical metric $d_c$ and fix $\epsilon$ in the visual metric $d_v=d_\epsilon$.
	For any small $s$, say $s=1$ for simplicity, let $R$ be as in Lemma \ref{sR} and let $$B=(1-2\epsilon')e^{-\epsilon(R-\frac12A+2\delta+s)}.$$
	Now for any $x,y\in \partial X$ with $d_v(x,y)\le B$, we have $$d_\epsilon(x,y)<(1-2\epsilon')e^{-\epsilon(R-\frac12A+2\delta+s)}.$$ Therefore
	\begin{align*}
	\rho_\epsilon(x,y)&<e^{-\epsilon(R-\frac12A+2\delta+s)},\\
	(x|y)&>R-\frac12A+2\delta+s,\\
	R-\frac12d(x(R),y(R))&>R-\frac12A+2\delta+s-2\delta-s,\\
	d(x(R),y(R))&<A.
	\end{align*}
	
	By Lemma \ref{claim}\ref{claim1}, there exists $R_1>R$ satisfying $d(x(R_1),y(R_1))=A$. Hence $d_c(x,y)=\frac1{R_1}$. By \eqref{1} we have
	\begin{align}
	(x|y)-2\delta-s&<R_1-\frac12A<(x|y)\nonumber,\\
	\frac1{(x|y)+\frac A2}&<\frac1{R_1}<\frac1{(x|y)+\frac A2-2\delta-s}.\label{2}
	\end{align}
	We let $d=d_v(x,y)$ for simplicity. Note that $(1-2\epsilon')e^{-\epsilon(x|y)}\le d\le e^{-\epsilon(x|y)}$, hence we have
	\begin{equation}\label{3}
	\frac{\ln d}{-\epsilon}+\frac{\ln(1-2\epsilon')}\epsilon\le (x|y)\le \frac{\ln d}{-\epsilon}.
	\end{equation}
	Combine \eqref{2} and \eqref{3}, we get 
	$$\frac1{\frac{\ln d}{-\epsilon}+\frac A2}<d_c(x,y)<\frac1{\frac{\ln d}{-\epsilon}+\frac A2-2\delta-s+\frac{\ln(1-2\epsilon')}{\epsilon}}.$$
	
	Recall that we let $s=1$ for simplicity, hence we finish the proof by letting $a=\frac{1}{\epsilon}$, $b=\frac A2$ and $k=2\delta+1-\frac1\epsilon\ln(1-2\epsilon')$.
\end{proof}

\begin{corollary}\label{cdim}
	Let X be a $\delta$-hyperbolic CAT(0) space, 
	then we have $$cdim_c(\partial X) \le cdim_v(\partial X).$$
\end{corollary}

\begin{proof}
	Let $\mathcal{U}$ be any cover of $\partial X$ realizing $cdim_v(\partial X)$. Without loss of generality, we may assume  
	\begin{align*}
	&mesh(\mathcal{U})=\lambda\le \min\{\frac1{e^2}, B \},\quad \text{and}\\
	&\mathcal{L}(\mathcal{U})\ge  c\lambda
	\end{align*}
	where $B$ is as in Proposition \ref{lem2} and $0<c\le 1$ is fixed.
	
	Let $a,b,k>0$ be as in Proposition \ref{lem2}, and let
	\begin{align*}
	&f_1(x)=\frac1{-a\ln x+b},\\
	&f_2(x)=\frac1{-a\ln x+b-k}.
	\end{align*}
	Fix $A$ in the definition of the conical metric. It's easy to check that 
	

	\begin{align*}
	\frac1{f_1}-\frac1{f_2}=k,\quad\quad
	\frac{f_2-f_1}{f_1f_2}=k
	\end{align*}
	Hence $$\frac{f_2}{f_1}-1=kf_2.$$
	When $x$ is small, $f_2(x)<1$. Hence
	$$\frac{f_2(x)}{f_1(x)}=1+kf_2(x)<1+k.$$
	
	By Proposition \ref{lem2} and Lemma \ref{lem1}, we know 
	\begin{align*}
	&mesh_M(\mathcal{U})< f_2(\lambda),\quad\text{and}\\
	&\mathcal{L}_M(\mathcal{U}) > f_1(c\lambda)\ge cf_1(\lambda)
	\end{align*}
	Therefore,
	$$Capacity_M(\mathcal{U})=\frac{\mathcal{L}_M(\mathcal{U})}{mesh_M(\mathcal{U})}>\frac{f_1(c\lambda)}{f_2(\lambda)}\ge \frac{cf_1(\lambda)}{f_2(\lambda)}>\frac c{1+k}.$$
	This means $\mathcal{U}$ also satisfies the condition of capacity dimension with the conical metric $d_c$, therefore $cdim_c(\partial X)\le cdim_v(\partial X)$.
	
\end{proof}

\begin{theorem}
	Let X be a cobounded $\delta$-hyperbolic CAT(0) proper geodesic space. Let $cdim_v(\partial X)$ be the capacity dimension of the boundary of $X$ with the visual metric and $cdim_c(\partial X)$ be the capacity dimension with the conical metric. Then we have $cdim_v(\partial X) = cdim_c(\partial X)$. In particular, for a hyperbolic CAT(0) group G, we have $cdim_v(\partial G) = cdim_c(\partial G)$.
\end{theorem}

\begin{proof}
	By Corollary \ref{cdim}, we know
	$$cdim_c(\partial X)\le cdim_v(\partial X).$$ By Theorem \ref{topequalcap} we know that 
	$$cdim_v(\partial X)= dim(\partial X).$$
	Since $\partial X$ is a compact metric space, by definition \ref{topologicaldimension}, any cover of $\partial X$ realizing the capacity dimension also satifies the condition for topological dimension, therefore we have
	$$dim(\partial X)\le cdim_c(\partial X).$$
	Combining all the (in)equalities above gives
	$$dim(\partial X)\le cdim_c(\partial X) \le cdim_v(\partial X) = dim(\partial X).$$
	This proves $cdim_c(\partial X) = cdim_v(\partial X)$.
	
\end{proof}

\section{Capacity dimension of the boundary of buildings }

We will prove the following theorem.
\begin{theorem}\label{thm}
	The capacity dimension of the boundary of any nonspherical building is equal to the capacity dimension of the boundary of an apartment in the building.
\end{theorem}

We will need the following lemma. Recall that $\rho:\Delta\rightarrow\Sigma$ is the apartment retraction.

\begin{lemma}(Lemma 1 in \cite{DS})
	For any $N>0$ there exists $M>0$ such that if $U$ is a subset of $W$ of $\delta$-diameter $\le N$, then any gallery-connected component $V$ of $\rho^{-1}(U)$ has $\delta$-diameter $\le M$.
\end{lemma}

To use the lemma with our settings, we will modify the proof a little to get the following lemma. The sketch of the proof will be given after the proof of the main theorem.
\begin{lemma}\label{lem}
	Let $\Delta$ be the Davis realization of a building and $\Sigma$ be the Davis complex of $W$. Let $\rho:\Delta\rightarrow\Sigma$ be the retraction with base chamber $C$. For any $U\subset \Sigma$ with $diam(U)\le R$, any path-connected component of $\rho^{-1}(U)$ has diameter at most $SR+M$, where $S$ and $M$ are constants depending only on the diameter of a chamber.
\end{lemma}

\begin{remark}
	The explicit upper bound in the lemma above is $2R+2D+M$, where $M=$ the diameter of a chamber, $D=\frac{Rr}{\epsilon}$, $r\ge 2M$ fixed, and $\epsilon=$Lebesgue number of a particular cover on $B_r(p)$, $p$ being the barycenter of the base chamber $B$.
\end{remark}

\begin{proof}[\bf Proof of Theorem \ref{thm}]
	
	Suppose $cdim(\partial\Sigma)=n$, since $\partial\Sigma$ embeds in $\partial\Delta$, we know $cdim(\partial \Delta)\ge n$. Hence, we need only to show the other direction. We will prove the theorem in four steps. In step 1, we set up the cover for the boundary of the building $\partial\Delta$. In step 2 and 3, we prove the bounds for the Lebesgue number and size of the cover. In step 4, we argue that the capacity dimension of $\partial\Delta$ is less than or equal to $n$.
	
	\emph{Step 1: } We fix an apartment $\Sigma$ and choose a base chamber $C\subset \Sigma$. Let $\rho$ be the retraction $\rho=\rho_{\Sigma,C}: \Delta\rightarrow \Sigma$. Let $p$ be a fixed point in $C$ and let it be the base point of $\Delta$ as a CAT(0) space. Then by Proposition \ref{prop1} we have  $d(p,x)=d(p,\rho(x))$ for any $x\in\Delta$.
	
	Fix $A$ for the conical metric on $\partial\Sigma$. Let $\{\mathcal{U}_i\}_{i=1}^{n+1}$ be any cover of $\partial\Sigma$ realizing the capacity dimension $cdim(\partial\Sigma)$, i.e. $n+1$ families of $L$-separated sets that are $cL$-bounded on $\partial\Sigma$. We may assume that $c>1$ by enlarging $c$. Therefore $cL>L$. 
	
	For any $r\in U\in \mathcal{U}_i$ for some $i$, $r: [0,\infty)\rightarrow \Sigma$ is a geodesic ray with $r(0)=p$. Hence $U$ can be viewed as a cone in $\Sigma$ consisting of geodesic rays. We are going to intersect each $U$ as a cone with two spheres of radius $\frac 1L$ and $\frac 1{cL}$ separately, as shown in Figure \ref{Figure 1}. More specifically, for any $U\in\mathcal{U}_i$, let 
	\begin{align*}
	&V_U=\left\{r\left(\frac1L\right )|r\in U\right\}\\
	&W_U=\left\{r\left(\frac1{cL}\right)|r\in U\right\}
	\end{align*}
	and
	\begin{align*}
	&\mathcal{V}_i=\left\{V_U| U\in \mathcal{U}_i\right\}\\
	&\mathcal{W}_i=\left\{W_U| U\in \mathcal{U}_i\right\}.
	\end{align*}
	
	\begin{figure}
		\begin{center}
			\includegraphics[scale=0.2]{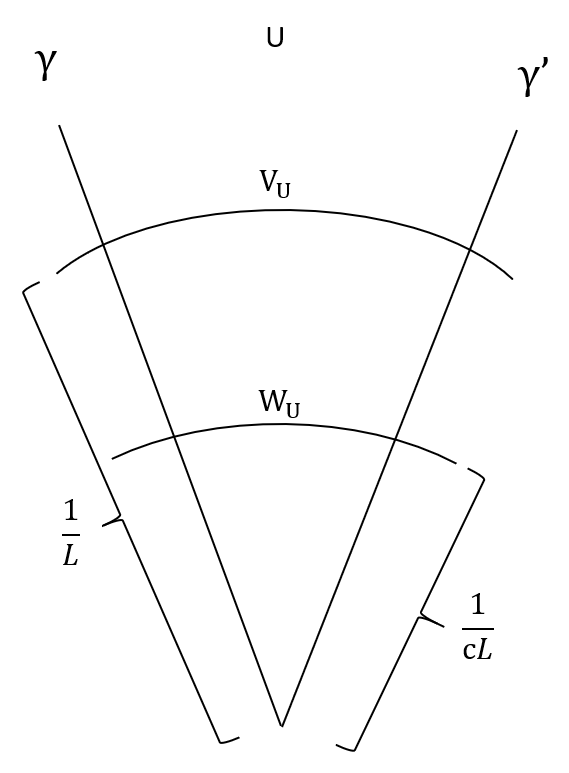}
			\caption{Intersections of $U$ with two spheres.\label{Figure 1}}
		\end{center}
	\end{figure}

	Therefore, if we let $S_{\frac1L}(p)\subset \Sigma$ denote the sphere of radius $\frac 1L$ in $\Sigma$ and $S_{\frac1{cL}}(p)\subset \Sigma$ denote the sphere of radius $\frac 1{cL}$, then  $\{\mathcal{V}_i\}_{i=1}^{n+1}$ is a cover of $S_{\frac1L}(p)$ and
	$\{\mathcal{W}_i\}_{i=1}^{n+1}$ is a cover of $ S_{\frac1{cL}}(p)$. In particular, by definition of the conical metric, $\mathcal{V}_i$ is $A$-separated and $\mathcal{W}_i$  is $A$-bounded for each $i$.

	To get the cover on $\partial\Delta$, we would like to pull back $\mathcal{V}_i$ through the retraction $\rho$. However, a problem here is, when a set  $V\in \mathcal{V}_i$ is very close to the boundary of a chamber, two different components of the preimage in $\rho^{-1}(V)$ can be very close, which is not ideal since we want the cover to be separated. We need some work to solve the problem, as shown in Figure \ref{Figure 2}.
	
	\begin{figure}
		\begin{center}
			\includegraphics[scale=0.2]{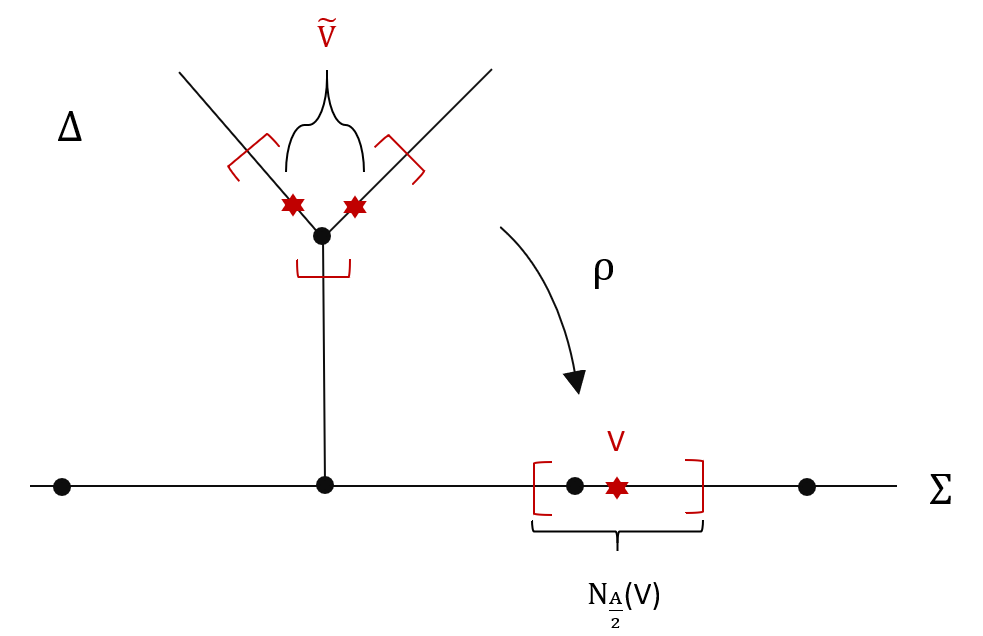}
			\caption{In the case of a tree, the set $V$ is a point. When $V$ is close to the endpoint, its preimages can be close. So instead, we take the preimage of neighborhood of $V$, and let all preimages of $V$ in one connected component be one set $\tilde{V}$.\label{Figure 2}}
		\end{center}
	\end{figure}
	
	Let $N_{\frac A2}(V)$ be the $\frac A2$ neighborhood of $V\in \mathcal{V}_i$ in $\Sigma$. Let $C_V$ be the set of all connected components of $\rho^{-1}(N_{\frac A2}(V))$ in $\Delta$.
	
	For each $K\in C_V$, let 
	$$\tilde{V}_K=\rho^{-1}(V)\bigcap K $$
	and let 
	$$\tilde{\mathcal{V}_i} = \left\{\tilde{V}_K \left| K\in C_V,  V\in \mathcal{V}_i\right.\right\}.$$
	
	Note that by Proposition \ref{prop1}, $\{\tilde{\mathcal{V}}_i\}_{i=1}^n$ is a cover of the sphere of radius $\frac 1L$ in $\Delta$, denoted by $\tilde{S}_{\frac 1L}(p)$. We will use $\tilde{\mathcal{V}}_i$ to generate the cover of $\partial\Delta$. More specifically, for each $\tilde{V}\in\tilde{\mathcal{V}}_i$, let 
	$$\tilde{U}_{\tilde{V}}=\left\{\text{ geodesic rays }\gamma  \left|  \gamma\left(\frac 1L\right)\in \tilde{V}\right.\right\}$$
	and
	$$\tilde{\mathcal{U}}_i=\left\{ \tilde{U}_{\tilde{V}}| \tilde{V}\in\mathcal{V}_i \right\}.$$
	In other words, $\tilde{U}_{\tilde{V}}$ is the set of geodesic rays whose intersections with $\tilde{S}_{\frac 1L}(p)$ are in $\tilde{V}$. Therefore, $\bigcup_{i=1}^{n+1}\tilde{U}_i$ is a cover of $\partial\Delta$. We need to show this cover satisfies the condition for capacity dimension, i.e. a upper bound of the mesh and a lower bound of the Lebesgue number. We will evaluate these two bounds on the two spheres separately.
	
	
	\emph{Step 2:} Now we do the estimate on $\tilde{S}_{\frac 1L}(p)$, which later will give a lower bound on the Lebesgue number of $\bigcup_{i=1}^{n+1}\tilde{U}_i$. More specifically, we prove the following lemma.
	\begin{lemma}
		$\tilde{\mathcal{V}_i}$ is $\frac A2$-separated on $\tilde{S}_{\frac 1L}(p)$ for each $i$.
	\end{lemma}
	
	\begin{proof}	
		For any $x\in \tilde{V_1}\in \tilde{\mathcal{V}}_i$ and $y\in \tilde{V_2}\in\tilde{\mathcal{V}}_i$, since $\rho: \tilde{V}\rightarrow V$ is onto, there exists $V_1,V_2\in \mathcal{V}_i$ such that $\rho(\tilde{V_1})=V_1$ and $\rho(\tilde{V_2})=V_2$. There are two cases.
		
		Case 1: $V_1\neq V_2$, then by proposition \ref{prop1}, we have 
		$$d(x,y)\ge d(\rho(x),\rho(y))\ge A>\frac A2.$$
		
		Case 2: $V_1=V_2 = V$, then since $\tilde{V_1}$ and $\tilde{V_2}$ belongs to different connected components of $\rho^{-1}(N_{\frac A2}(V))$, there exists a point $q\in[x,y]$\textbackslash $\rho^{-1}(N_{\frac A2}(V))$, where $[x,y]$ is the geodesic connecting $x$ and $y$. Therefore,
		$$d(x,y)=d(x,q)+d(q,y)\ge d(\rho(x),\rho(q))+d(\rho(q),\rho(y))\ge\frac A2 +\frac A2>\frac A2.$$
		
		Thus $\tilde{\mathcal{V}_i}$ is $\frac A2$-separated.
	\end{proof}
	
	\emph{Step 3:} Now we do the estimate on $\tilde{S}_{\frac 1{cL}}(p)$, which later will give an upper bound on the mesh of $\bigcup_{i=1}^{n+1}\tilde{U}_i$. More specifically, for each $\tilde{U}\in\tilde{\mathcal{U}}_i$, let 
	$$\tilde{W}_{\tilde{U}}=\left\{\tilde{\gamma}\left(\frac1{cL}\right) \left| \tilde{\gamma}\in \tilde{U}\right.\right\}$$
	and
	$$\tilde{\mathcal{W}}_i = \left\{ \tilde{W}_{\tilde{U}} | \tilde{U}\in\tilde{\mathcal{U}}_i \right\}.$$
	Then $\bigcup_{i=1}^{n+1}\tilde{\mathcal{W}}_i$ is a cover of $\tilde{S}_{\frac 1{cL}}(p)$. We want to prove the following lemma.
	\begin{lemma}
		For each $i$ and each $\tilde{W}\in\tilde{\mathcal{W}}_i$, we have $diam(\tilde{W})\le S'A+M$ for some constants $S'$ and $M$ that depend only on the diameter of the base chamber $C$.
	\end{lemma}
	
	\begin{proof}
		For each $\tilde{U}\in\tilde{\mathcal{U}}_i$, let 
		$$c\tilde{U}=\bigcup_{r\in\tilde{U}}r([0,\infty))\subseteq \Delta$$
		be the cone corresponding to $\tilde{U}$ in $\Delta$. The intersection of $c\tilde{U}$ with $\tilde{S}_{\frac 1L}(p)$ is a set $\tilde{V}\in\tilde{\mathcal{V}}_i$, and its intersection with $\tilde{S}_{\frac 1{cL}}(p)$ is a set $\tilde{W}\in\tilde{\mathcal{W}}_i$. We will show $diam(\tilde{W})$ is bounded.
		
		Let $V=\rho(\tilde{V})$ and $W=\rho(\tilde{W})$. Then $V\in\mathcal{V}_i$. Notice $\rho$ maps a geodeisc ray in $\Delta$ to a geodesic ray in $\Sigma$, hence by definition of $\tilde{\mathcal{W}}_i$ and $\mathcal{W}_i$, we also have $W\in\mathcal{W}_i$. Recall that $\tilde{V}$ is contained in a connected component $K$ of $\rho^{-1}(N_{\frac A2}(V))$. Let
		$$N(\tilde{W}) = \left\{ \tilde{\gamma}\left(\frac1{cL}\right)\left| \tilde{\gamma}\left(\frac 1L\right)\in K\right. \right\}$$
		and
		$$N(W)=\left\{ \gamma\left(\frac1{cL}\right)\left| \gamma\left(\frac 1L\right)\in N_{\frac A2}(V)\right. \right\}.$$
		Then $\tilde{W}\subset N(\tilde{W})$. Since $K$ is connected, $N(\tilde{W})$ is also connected by the uniqueness of geodesics in CAT(0) spaces. In addition, $\rho(N(\tilde{W}))=N(W)$ since $\rho$ maps a geodesic ray in $\Delta$ to a geodesic ray in $\Sigma$. Therefore, we can conclude that $N(\tilde{W})$ is contained in a connected component of $\rho^{-1}(N(W))$.

		By Lemma \ref{ratio}, we have
		\begin{align*}
		\frac B{\frac A2}&\le \frac{\frac 1L}{\frac 1{cL}}\\
		B&\le \frac A{2c}
		\end{align*}
		Where $B$ is shown in Figure \ref{Figure 3}. Hence we have
		$$N(W)\subset N_{\frac A{2c}}(W)$$
		where $N_{\frac A{2c}}(W)$ is the $\frac A{2c}$ neighborhood of $W$. Hence
		$$diam(N(W))\le diam(N_{\frac A{2c}}(W))\le A+2\cdot\frac A{2c}=\frac{c+1}{c}\cdot A.$$
		
		\begin{figure}
			\begin{center}
				\includegraphics[scale=0.2]{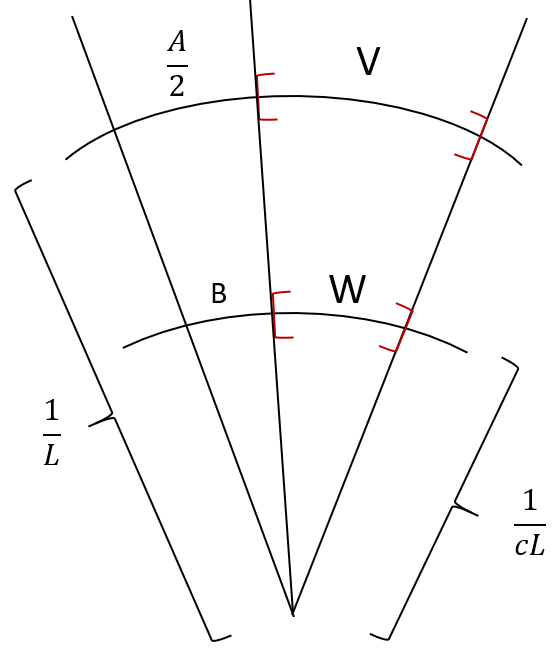}
				\caption{Relation between $N_{\frac A2}(V)$ and $N(W)$.\label{Figure 3}}
			\end{center}
		\end{figure}
		
		By Lemma \ref{lem}, we have
		$$diam(N(\tilde{W}))\le S\cdot\frac{c+1}{c}\cdot A + M.$$
		In particular, let $S'=S\cdot\frac{c+1}c$, we have $diam(\tilde{W})\le S'A+M.$
	\end{proof}
	
	
	\emph{Step 4:}	
	Finally, we estimate the bounds of $\bigcup_{i=1}^{n+1}\tilde{\mathcal{U}}_i$ on $\partial\Delta$. We still use $A$ for the conical metric on $\partial\Delta$. Recall $\bigcup_{i=1}^{n+1}\tilde{\mathcal{U}}_i$ is a cover of $\partial\Delta$. We prove the following lemma.
	\begin{lemma}\label{lastlem}
		For each $i$, $\tilde{\mathcal{U}}_i$ is $\frac L2$-separated and $c'L$ bounded for some constant $c'$.
	\end{lemma}
	
	\begin{proof}
		For any $i$, suppose $\tilde{\mathcal{U}}_i$ is $L'$-separated, that is, let $\frac 1{L'}=\inf d(\gamma,\gamma')$ where the infimum is taken over all $\gamma$ and $\gamma'$ that belong to two different sets in  $\tilde{\mathcal{U}}_i$. Then as shown in Figure \ref{Figure 4} left, by definition of the conical metric and Lemma \ref{ratio} we have
		$$\frac{\frac A2}{A}\le \frac{\frac 1L}{\frac 1{L'}},\quad\text{hence } L'\ge\frac L2.$$
		
		\begin{figure}
			\begin{center}
				\includegraphics[scale=0.2]{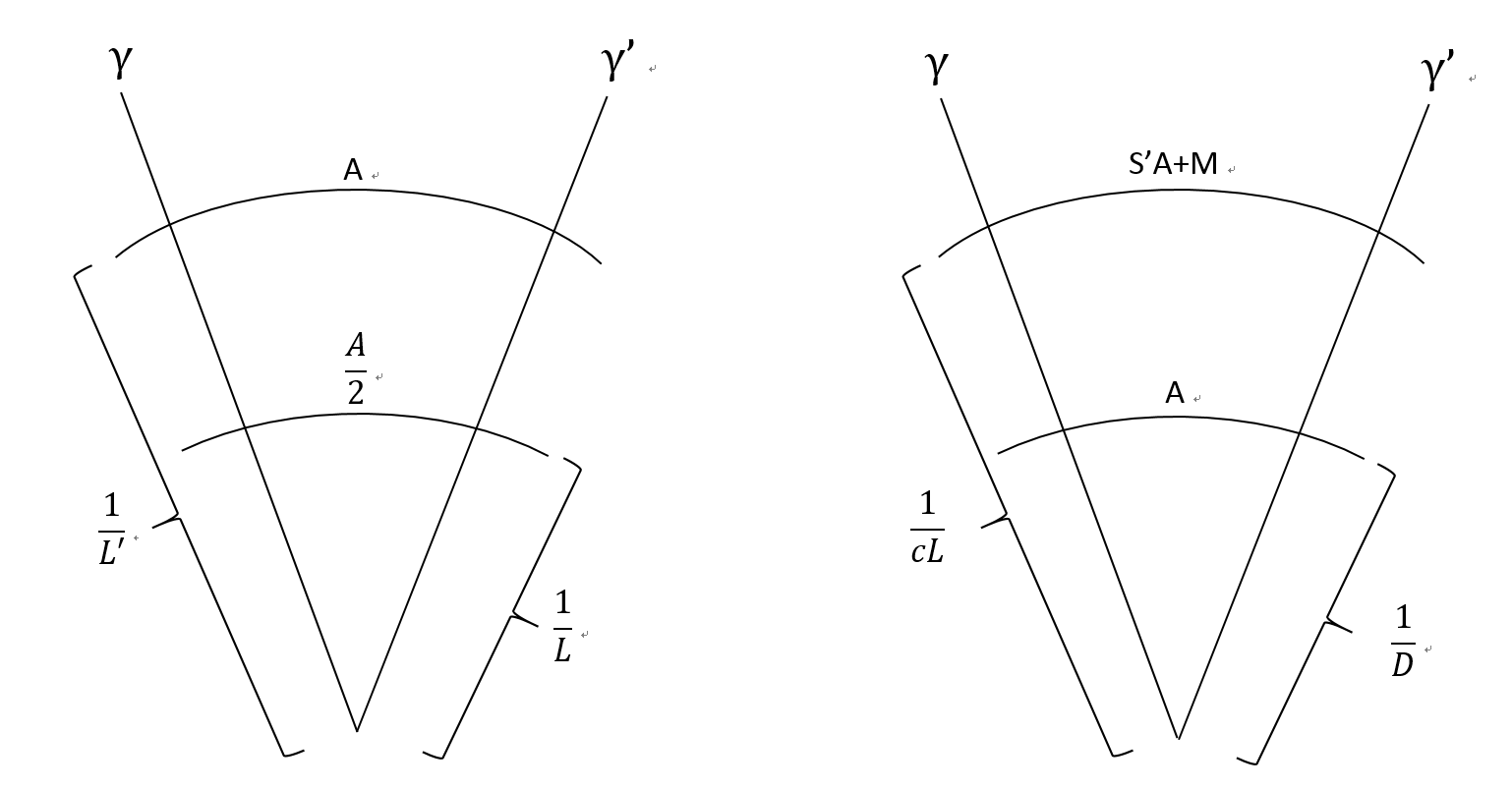}
				\caption{Intersections of $U$ with two spheres.\label{Figure 4}}
			\end{center}
		\end{figure}
		
		Suppose $\tilde{\mathcal{U}}_i$ is $D$ bounded. Then similarly, as shown in Firgure \ref{Figure 4} right, by Lemma \ref{ratio} we have
		$$\frac{A}{S'A+M}\le\frac{\frac 1D}{\frac1{cL}},\quad\text{hence } D\le cL\cdot\frac{S'A+M}{A}.$$
		Let $c'=c\cdot\frac{S'A+M}{A}$, then $D\le c'L$.
	\end{proof}
	
	Finally, by definition of capacity dimension, Lemma \ref{lastlem} implies $cdim(\partial\Delta)\le n$. Hence $cdim(\partial\Delta)= n$.
\end{proof}


Now we prove Lemma \ref{lem}. 
\begin{proof}[\bf Proof of Lemma \ref{lem}]
	Let $W$ be a Coxeter group, $\Sigma$ be the Davis realization of $W$ and $C$ be the base chamber. Let $p$ be the barycenter of $C$. There are two lemmas proved in \cite{DS}.
	\begin{lemma}\label{lem01} (Lemma 2 in \cite{DS})
		For any $R>0$ there exists $D=D(R)$ such that for any subset $U\subset \Sigma$ of diameter $R$ satisfying $d(C,U)>D$ there exists a codimension-one face of $C$ such that the wall containing that face separates $C$ from $U$.
	\end{lemma}
	
	In the lemma above, $D=\frac{Rr}{\epsilon}$, where $r$ is a fixed number greater than twice the diameter of a chamber, $\epsilon$ is the Lebesgue number of a particular cover of $S_r(p)$.
	
	Recall $d$ is the CAT(0) distance and $\delta$ be the gallery distance. For $X\subset \Sigma$, denote by $T(X)$ the union of all chambers that intersect $X$. Let $U$ be a subset of $\Sigma$ of $\delta$-diameter $\le N$. There exists $R>0$ depending only on $N$ (and $\Sigma$) such that the $d$-diameter of $T(U)$ is $\le R$.  Iterated application of Lemma \ref{lem01} gives us a minimal gallery $\gamma=(C,w_1C,\ldots,w_kC)$ such that the wall between $w_iC$ and $w_{i+1}C$ separates $w_iC$ from $T(U)$ and $d(w_kC,T(U))\le D$. Note that this separation property implies that every chamber which meets $U$ can be joined to $C$ by a minimal gallery which is extending $\gamma$.
	
	\begin{lemma}\label{lem02} (Lemma 3 in \cite{DS})
		Let $U$ and $\gamma$ be given as above. Recall $C$ is the base chamber which is fixed pointwise by $\rho$. For any chamber $C'\subset\Delta$ meeting $\rho^{-1}(U)$ there is a minimal gallery from $C$ to $C'$ whose $\rho$ projection extends $\gamma$.
		
		For any gallery-connected component $V\subset\Delta$ of $\rho^{-1}(U)$ there exists a chamber $E\in \rho^{-1}(w_kC)$ such that any minimal gallery from $C$ to a chamber in $V$ whose $\rho$-projection prolongs $\gamma$ passes through $E$.
	\end{lemma}
	
	Recall $diam(U)\le R$ and $M$ is the diameter of a chamber. Let $\tilde{U}$  be any connected component of $\rho^{-1}(U)$, then $\tilde{U}$ is contained in a gallery-connected component $V$ of $\rho^{-1}(U)$. By Lemma \ref{lem02}, we know 
	$d(\tilde{U}, E)\le D.$
	Hence $$diam(\tilde{U})\le 2R+2D+M$$
	where $D=\frac {Rr}{\epsilon}$. The proof is finished by letting $S=2(1+\frac{r}{\epsilon})$.
\end{proof}

As a quick corollary, we have the following:
\begin{corollary}
	Let $X$ be the product of $n$ trees. Then $cdim({\partial}X)=n-1$.
\end{corollary}
Previously, the capacity dimension of the boundary of product of trees is only known to be either $n$ or $n-1$.

\section{Further Questions}

As mentioned in the introduction, the ultimate goal is to prove that CAT(0) groups have finite asymptotic dimension. Moran proved that the boundaries of CAT(0) groups have finite capacity dimension. Therefore we hope to prove the inequality in Theorem \ref{asdimcdiminequality}:
\begin{equation}\label{theinequality}
asdim X\le cdim\partial X+1.
\end{equation}
One possible way is to use the Hurewicz-type mapping theorem, which is proved in \cite{BD}.
\begin{theorem}\label{Hurewicz}
	Lef $f:X\rightarrow Y$ be a Lipschitz map from a geodesic metric space $X$ to a metric space $Y$. Suppose that for every $R>0$ the set family $\{f^{-1}(B_R(y))\}_{y\in Y}$ satisfies the inequality $asdim\le n$ uniformly. Then $asdim X\le asdim Y+n$.
\end{theorem}

To use the Theorem, let $X$ be the CAT(0) space with basepoint $x_0$, let $Y$ be $\mathbb{R}$, and for any $x\in X$, let $f(x)=d(x,x_0)$. Then for any $B_R(y)\subset Y$, $f^{-1}(B_R(y))$ is a circular ring region of width $R$ in $X$, denoted by $CR_D$, as shown in Figure \ref{futurequestion}. According to Theorem \ref{Hurewicz}, we want to show that $asdim CR_D\le n$, where $n$ is the capacity dimension of $\partial X$. In particular, the mesh and the Lebesgue number of the cover realizing the asymptotic dimension should not depend on $D$.

\begin{figure}
	\begin{center}
		\includegraphics[scale=0.22]{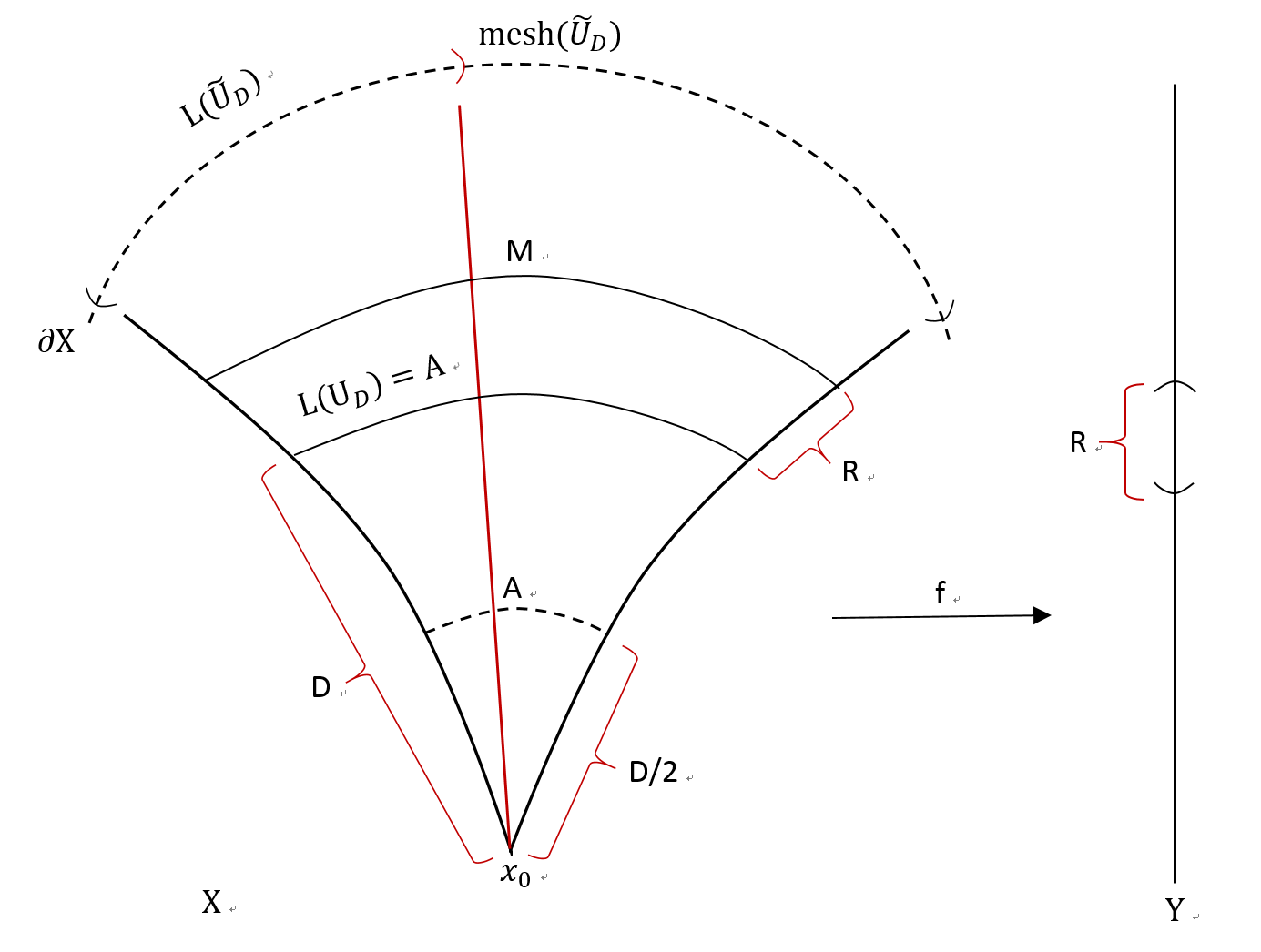}
		\caption{Intersections of $U$ with two spheres.\label{futurequestion}}
	\end{center}
\end{figure}

The natural idea is for each $D$, we find cover $\mathcal{\tilde{U}}_D$ of $\partial X$ of particular mesh, realizing the capacity dimension of $\partial X$. Thinking of each set in $\mathcal{\tilde{U}}_D$ as a set of geodesic rays, we intersect each set with $CR_D$ to get a cover $\mathcal{U}_D$ of $CR_D$. We hope to choose $mesh(\mathcal{\tilde{U}}_D)$ and $\mathcal{L}(\mathcal{\tilde{U}}_D)$ properly so that there exist $M$ and $\lambda$, for any $D>0$, we have $mesh(\mathcal{U}_D)\le M$ and $\mathcal{L}(\mathcal{U}_D)\ge \lambda$.

Assume $cdim(\partial X)=n$, then by definition, there exists a constant $0<c\le1$ and $\lambda_0$, such that for any $\lambda<\lambda_0$, there exists a cover $\mathcal(\tilde{U})$ of $\partial X$ with $mesh(\tilde{U})\le\lambda$ and $\mathcal{L}(\tilde{U})\ge c\lambda$. For simplicity, let's assume $c=\frac12$, and let $A=\lambda$ in the definition of the conical metric on $\partial X$. Fix $D>0$, choose a cover $\mathcal{\tilde{U}}_D$ on $\partial X$ realizing $cdim \partial X$ with $mesh(\mathcal{\tilde{U}}_D)\le \frac 2D$ and $\mathcal{L}(\mathcal{\tilde{U}}_D)\ge\frac 1D$. Again, for simplicity, we assume the equality holds, i.e.  $mesh(\mathcal{\tilde{U}}_D)= \frac 2D$ and $\mathcal{L}(\mathcal{\tilde{U}}_D)=\frac 1D$.

By definition of the conical metric, we automatically have $\mathcal{L}(\mathcal{U}_D)=A$ for all $D$. Notice $mesh(\mathcal{U}_D)\le M+2R$, hence to get an upper bound on  $mesh(\mathcal{U}_D)$, we need only find an upper bound for $M$. However, the only inequality we have about $M$ is 
$$\frac AM\le \frac{\frac D2}{D+R},$$
which only gives an lowerbound for $M$. 

If $X$ is Euclidean, then the inequality above becomes equality. Therefore the inequality \ref{theinequality}  is proved. However, if the curvature of $X$ is very negative, for example, when $X$ is a tree, then as $D$ getting larger, $M$ can be arbitrarily large. This is the main difficulty of the method. However, when the curvature is very negative, $X$ normally have very nice properties. For example, the inequality \ref{theinequality} for the tree is pretty easy to prove. 

One possible method is to try to find a balance between the two extreme cases. When curvature is close to $0$, use Euclidean-like properties, while when curvature is very negative, use tree-like properties. We hope this can be a valuable idea for further study on the asymptotic dimension of CAT(0) spaces.


\bibliographystyle{plain}
\bibliography{references}

\begin{thebibliography}{10}

\bibitem{AB}
Peter Abramenko and Kenneth~S. Brown.
\newblock {\em Buildings}, volume 248 of {\em Graduate Texts in Mathematics}.
\newblock Springer, New York, 2008.
\newblock Theory and applications.

\bibitem{BD}
G.~C. Bell and A.~N. Dranishnikov.
\newblock A {H}urewicz-type theorem for asymptotic dimension and applications
  to geometric group theory.
\newblock {\em Trans. Amer. Math. Soc.}, 358(11):4749--4764, 2006.

\bibitem{BH}
Martin~R. Bridson and Andr\'{e} Haefliger.
\newblock {\em Metric spaces of non-positive curvature}, volume 319 of {\em
  Grundlehren der Mathematischen Wissenschaften [Fundamental Principles of
  Mathematical Sciences]}.
\newblock Springer-Verlag, Berlin, 1999.

\bibitem{B}
S.~V. Buyalo.
\newblock Asymptotic dimension of a hyperbolic space and the capacity dimension
  of its boundary at infinity.
\newblock {\em Algebra i Analiz}, 17(2):70--95, 2005.

\bibitem{BL}
S.~V. Buyalo and N.~D. Lebedeva.
\newblock Dimensions of locally and asymptotically self-similar spaces.
\newblock {\em Algebra i Analiz}, 19(1):60--92, 2007.

\bibitem{D}
Michael~W. Davis.
\newblock Buildings are {${\rm CAT}(0)$}.
\newblock In {\em Geometry and cohomology in group theory ({D}urham, 1994)},
  volume 252 of {\em London Math. Soc. Lecture Note Ser.}, pages 108--123.
  Cambridge Univ. Press, Cambridge, 1998.

\bibitem{DS}
J.~Dymara and T.~Schick.
\newblock Buildings have finite asymptotic dimension.
\newblock {\em Russ. J. Math. Phys.}, 16(3):409--412, 2009.

\bibitem{GH}
\'{E}. Ghys and P.~de~la Harpe, editors.
\newblock {\em Sur les groupes hyperboliques d'apr\`es {M}ikhael {G}romov},
  volume~83 of {\em Progress in Mathematics}.
\newblock Birkh\"{a}user Boston, Inc., Boston, MA, 1990.
\newblock Papers from the Swiss Seminar on Hyperbolic Groups held in Bern,
  1988.

\bibitem{Gromov}
M.~Gromov.
\newblock Hyperbolic groups.
\newblock In {\em Essays in group theory}, volume~8 of {\em Math. Sci. Res.
  Inst. Publ.}, pages 75--263. Springer, New York, 1987.

\bibitem{Moran}
Molly~A. Moran.
\newblock {\em On the dimension of group boundaries}.
\newblock PhD thesis, University of Wisconsin-Milwaukee, 2015.

\bibitem{Moran2}
Molly~A. Moran.
\newblock Metrics on visual boundaries of {$\rm CAT(0)$} spaces.
\newblock {\em Geom. Dedicata}, 183:123--142, 2016.

\bibitem{M}
James~R. Munkres.
\newblock {\em Topology}.
\newblock Prentice Hall, Inc., Upper Saddle River, NJ, 2000.
\newblock Second edition of [ MR0464128].

\bibitem{OS}
Damian Osajda and Jacek \'{S}wi\c{a}tkowski.
\newblock On asymptotically hereditarily aspherical groups.
\newblock {\em Proc. Lond. Math. Soc. (3)}, 111(1):93--126, 2015.

\bibitem{Yu}
Guoliang Yu.
\newblock The {N}ovikov conjecture for groups with finite asymptotic dimension.
\newblock {\em Ann. of Math. (2)}, 147(2):325--355, 1998.

\end{thebibliography}

\end{document}